\documentclass[preprint,10pt]{elsarticle}
\usepackage{graphicx}
\usepackage{amsmath}
\usepackage{amsthm}
\usepackage{amsfonts}
\usepackage{amssymb, amscd}
\usepackage{mathrsfs}

\usepackage{color}

\usepackage[all]{xy}

\newtheorem{thm}{Theorem}[section]
\newtheorem{cor}[thm]{Corollary}
\newtheorem{lem}[thm]{Lemma}
\newtheorem{proposition}[thm]{Proposition}
\newtheorem{example}[thm]{Example}
\theoremstyle{definition}
\newtheorem{definition}[thm]{Definition}
\theoremstyle{remark}
\newtheorem{remark}[thm]{Remark}
\numberwithin{equation}{section}

\newcommand{\K}{\mathbb K}

\newcommand{\g}{ \mathfrak g}

\begin{document}
\begin{frontmatter}
\title
{Quadratic color Hom-Lie algebras}

\author[Sfax]{F. AMMAR}
\ead{Faouzi.Ammar@fss.rnu.tn}

\author[Mulhouse]{I. AYADI}
\ead{imen.ayadi@uha.fr }

\author[Sfax]{S. MABROUK}
\ead{Mabrouksami00@yahoo.fr}

\author[Mulhouse]{A. MAKHLOUF}%
\ead{Abdenacer.Makhlouf@uha.fr}

\address[Sfax]{
Universit\'{e} de Sfax,  Facult\'{e} des Sciences, Sfax Tunisia}%

\address[Mulhouse]{
Universit\'{e} de Haute Alsace,  Laboratoire de Math\'{e}matiques, Informatique et Applications, 
Mulhouse, France}%

\journal{arXiv}

\date{}
%
\begin{abstract}
The purpose of this paper is to  study quadratic  color Hom-Lie algebras.  We present some constructions of quadratic color Hom-Lie algebras which we use to provide several examples. We describe  $T^*$-extensions and  central extensions of color Hom-Lie algebras and establish some cohomological characterizations.\\ 
\end{abstract}
\end{frontmatter}

\section*{Introduction}

The aim of this  paper is to introduce and study quadratic color Hom-Lie algebras which are graded Hom-Lie algebras with $\varepsilon$-symmetric, invariant and nondegenerate bilinear forms. Color Lie algebras, originally introduced in \cite{rittenberg1978generalized} and \cite{rittenberg1978sequences}, can be seen as a direct generalization of Lie algebras bordering Lie superalgebras. The grading   is determined by an abelian group $\Gamma$ and the definition involves a bicharacter function. Hom-Lie algebras are a generalization of Lie algebras, where the classical Jacobi identity is twisted by a linear map. Quadratic  Hom-Lie algebras were studied in  \cite{benayadi2010hom}.\\

 $\Gamma$-graded Lie algebras with quadratic-algebraic structures, that is  $\Gamma$-graded Lie algebras provided with homogeneous, symmetric, invariant and nondegenerate bilinear forms,  have been extensively studied specially in the case where $\Gamma = {\mathbb Z}_2$ (see for example \cite{albuquerque2010odd,ayadi2010lie,bajo2007generalized,benamor1999double,benayadi2000quadratic,medina1985algebres,scheunert1979theory}). These algebras are called homogeneous (even or odd) quadratic Lie superalgebras. One of the fundamental results connected to  homogeneous quadratic Lie superalgebras is to give its inductive descriptions. The main tool used to obtain these inductive descriptions is to develop some concept of double extensions. This concept was introduced by Medina and Revoy (see \cite{medina1985algebres}) to give a classification of quadratic Lie algebras.  The concept of $T^*$-extension was introduced by  Bordemann \cite{Bordemann}. Recently  a generalization to the case of quadratic (even quadratic) color Lie algebras was obtained in \cite{qingcheng2008derivations} and \cite{wang2010non}. They mainly generalized the double extension notion and its inductive descriptions.\\
 
  Hom-algebraic structures appeared first as a generalization of Lie algebras in \cite{aizawa1991q,chaichian1990quantum,chaichian1990q} were the authors studied $q$-deformations of Witt and Virasoro algebras. A general study and construction of Hom-Lie algebras were considered  in \cite{hartwig2006deformations}, \cite{larsson2005quasi}. Since then, other interesting Hom-type algebraic structures of many classical structures were studied as Hom-associative algebras, Hom-Lie admissible algebras and more general $G$-Hom-associative algebras  \cite{makhlouf2007notes},  $n$-ary Hom-Nambu-Lie algebras  \cite{ammar2011quadratic}, Hom-Lie admissible Hom-coalgebras and Hom-Hopf algebras \cite{makhlouf2009hom}, Hom-alternative algebras, Hom-Malcev algebras  and Hom-Jordan algebras  \cite{elhamdadi2010deformations,makhlouf2009hom,Yau2012}.  Hom-algebraic structures  were extended to the case of $\Gamma$-graded Lie algebras by studying Hom-Lie superalgebras and Hom-Lie admissible superalgebras in \cite{ammar2010hom}. Recently, the study of Hom-Lie algebras provided with quadratic-algebraic structures was initiated by S. Benayadi and A. Makhlouf in \cite{benayadi2010hom} and our purpose in this paper is  to generalize this study to the case of color Hom-Lie algebras.\\

The paper is organized as follows: in Section 1 we give definition of quadratic color Hom-Lie algebras and we generalize to the case of quadratic color Hom-Lie algebras the result of \cite{yuan2010hom}. More precisely, we prove that the category of color Hom-Lie algebras (resp. the subcategory of quadratic color Hom-Lie algebras) is closed under self weak morphisms (resp. self symmetric automorphisms). In Section 2, we describe in one hand some constructions of quadratic color Hom-Lie algebras and in the other hand we  use them to provide examples. In Sections 3,  we  give some elements of representation theory of color Hom-Lie algebras and describe a semi-direct product.  Section 4 includes the main results dealing extensions and their relationships to cohomology. We introduce some elements of a cohomology of color Hom-Lie algebras controlling central extensions.  Moreover, we  establish the $T^*$-extension  for color Hom-Lie algebras. Finally, in Section 5 we construct  color Hom-Leibniz algebras arising from Faulkner construction introduced in \cite{figueroa2009deformations}.

\section{Quadratic color Hom-Lie algebras}

Throughout this paper  $\K$ denotes a commutative field of characteristic zero and  $\Gamma$ stands for an abelian group. A vector space $V$ is said to be a $\Gamma$-graded if we are given a family $(V_{\gamma})_{\gamma\in \Gamma}$ of vector subspace of $V$ such that $$V=\bigoplus_{\gamma\in \Gamma}V_{\gamma}.$$
An element $x\in V_{\gamma}$ is said to be homogeneous of degree $\gamma$. For simplicity  the degree of an element $x$ is denoted by $x$. If the base field is considered as a graded vector space, it is understood that the graduation of $\K$ is given by 
$${\K}_0 = \K  \ \ \ \mbox{and} \ \ \ {\K}_{\gamma}= \left\{0\right\}, \ \ \mbox{if}\ \ \gamma\in \Gamma\setminus \left\{0\right\}.$$ 
Let $V$ and $W$ be two $\Gamma$-graded vectors spaces. A linear map $f: V \longrightarrow W$ is said to be homogeneous of degree $\xi\in \Gamma$  if $f(x)$ is homogeneous of degree $\gamma + \xi$ whenever the element $x\in V_{\gamma}$. The set of all such maps is denoted  by ${(Hom(V,W))}_{\xi}$. It is a subspace of $Hom(V,W)$, the vector space of all linear maps from $V$ into $W$. \\

We mean by algebra (resp. $\Gamma$-graded algebra) $(A,.)$ a vector space (resp. $\Gamma$-graded vector space) with multiplication which  we denote by the concatenation and  such that $A_{\gamma}A_{\gamma'}\subseteq A_{\gamma+\gamma'}$, for all $\gamma, \gamma'\in \Gamma$. In the graded case, a map $f: A \longrightarrow B$, where $A$ and $B$ are $\Gamma$-graded algebras, is called a homomorphism of $\Gamma$-graded algebras if it is a homomorphism of algebras which is homogeneous of degree zero.\\

We mean by Hom-algebra (resp. $\Gamma$-graded Hom-algebra ) a triple  $(A,., \alpha)$ consisting of an  algebra (resp. $\Gamma$-graded algebra) with a  twist map $\alpha:A\rightarrow A$ (an even linear map).\\

For more detail about graded algebraic structures, we refer to \cite{scheunert1979generalized} \cite{scheunert1983graded}. In the following, we study a particular case of $\Gamma$-graded algebras which are quadratic color Hom-Lie  algebras.\\

\begin{definition}
Let $\Gamma$ be an abelian group. A map $\varepsilon:\Gamma\times\Gamma\rightarrow\K\setminus \left\{0\right\}$ is called a \textit{bicharacter} on ${\Gamma}$ if the following identities are satisfied
\begin{align}
\label{condi1-bicharacter}&\varepsilon(a,b)\varepsilon(b,a)=1,\\
\label{condi2-bicharacter}&\varepsilon(a,b+c)=\varepsilon(a,b)\varepsilon(a,c),\\
\label{condi3-bicharacter}&\varepsilon(a+b,c)=\varepsilon(a,c)\varepsilon(b,c),\ \ \forall a,b,c\in {\Gamma}.
\end{align}
\end{definition}

The definition above implies, in particular, the following relations
\begin{align*}
\varepsilon(a,0)=\varepsilon(0,a)=1,\ \varepsilon(a,a)=\pm1, \  \textrm{for\ all}\  a\in\Gamma.
\end{align*}
If $x$ and $x'$ are two homogeneous elements of degree $\gamma$ and $\gamma'$ respectively and $\varepsilon$ is a bicharacter, then we shorten the notation by writing $\varepsilon(x,x')$ instead of $\varepsilon(\gamma,\gamma')$. 

 Unless stated, in the sequel all the graded spaces are over the same abelian group $\Gamma$ and the bicharacter will be the same for all the structures. \\


%

\begin{definition}
A \textit{color Hom-Lie } algebra is a tuple $(\mathfrak{g},[\cdot,\cdot],\alpha,\varepsilon)$ consisting of a $\Gamma$-graded vector space $\mathfrak{g}$, a bicharacter  $\varepsilon$, an even bilinear map
$[\cdot,\cdot]:\mathfrak{g}\times\mathfrak{g}\rightarrow\mathfrak{g}$ $($i.e. $[\mathfrak{g}_a,\mathfrak{g}_b]\subset \mathfrak{g}_{a+b})$ and an even homomorphism $\alpha:\mathfrak{g}\rightarrow\mathfrak{g}$ such that for homogeneous elements $x,y,z\in\g$ we have 
\begin{align}
\label{sksymmetric}[x,y]=-\varepsilon(x,y)[y,x]&\ \ (\varepsilon-\textrm{skew\ symmetry}),\\
\label{JacobiId1}\circlearrowleft_{x,y,z}\varepsilon(z,x)[\alpha(x),[y,z]]=0&\ \ (\varepsilon-\textrm{Hom-Jacobi\ identity}),
\end{align}
where $\circlearrowleft_{x,y,z}$ denotes summation over the cyclic permutation on $x,y,z$. In particular, if $\alpha$ is a morphism of Lie algebras $($i.e. $\alpha\circ[\cdot,\cdot]=[\cdot,\cdot]\circ\alpha^{\otimes2})$, then we call   $(\mathfrak{g},[\cdot,\cdot],\varepsilon,\alpha)$ a multiplicative color Hom-Lie algebra.\\
\end{definition}

A $\Gamma$-graded subspace $I$ of $\g$ is said to be an ideal (resp. subalgebra) if $[I,\g]\subset I$ (resp. $[I,I] \subset I$).\\

We prove easily that the $\varepsilon-\textrm{Hom-Jacobi\ identity}$ is  equivalent to the following condition
\begin{align}\label{JacobiId2}
&[\alpha(x),[y,z]]-\varepsilon(x,y)[\alpha(y),[x,z]]=[[x,y],\alpha(z)].
\end{align}
We recover color Lie algebra when we have $\alpha=id_{\g}$. Color Lie algebra is a generalization of Lie algebra and Lie superalgebra  (if $\Gamma= \left\{0\right\}$, we have  ${\mathfrak g}= {\mathfrak g}_{0}$ is a Lie algebra and if $\Gamma= {\mathbb Z}_2=\left\{\bar{0},\bar{1}\right\}$ and $\varepsilon(\bar{1},\bar{1})= -1$, then ${\mathfrak g}$ is  a Lie superalgebra).\\

\begin{definition}
Let $(\mathfrak{g},[\cdot,\cdot],\alpha)$ and $(\mathfrak{g}',[\cdot,\cdot]',\alpha')$ be two color Hom-Lie algebras. An even linear map $f:\g\rightarrow\g'$ is called: 
\begin{enumerate}
\item[(i)] a \emph{weak morphism} of color Hom-Lie algebras if it satisfies $f([x,y])=[f(x),f(y)]'$, $\forall x,y\in \g$.
\item[(ii)] a \emph{morphism} of color Hom-Lie algebras  $f$ is a weak morphism and $f\circ\alpha=\alpha'\circ f$.
\item[(iii)] an \emph{automorphism} of color Hom-Lie algebras if $f$ is a bijective morphism of color Hom-Lie algebras.\\
\end{enumerate}
\end{definition}

In \cite{yuan2010hom}, the author applied the twisting principle introduced by Yau  and proved that starting from a color Lie algebra $(\mathfrak{g},[\cdot,\cdot], \varepsilon)$ and a weak morphism $\beta$ of $\g$ we obtain a color Hom-Lie algebra $(\mathfrak{g},[\cdot,\cdot]_\beta,\beta, \varepsilon)$, where $[\cdot,\cdot]_\beta=\beta\circ[\cdot,\cdot]$. In the following, we generalize the result obtained in \cite{yuan2010hom} in the following sense.\\


We consider the category of color Hom-Lie algebras consisting in color Hom-Lie algebras and morphisms of color Hom-Lie algebras. We give a color Hom-Lie version  of \cite[Theorem  2.5 ]{yuan2010hom}.\\

\begin{thm}\label{ConsWeakMorph}
The category of color Hom-Lie algebras is closed under self weak morphisms.\\
\end{thm}


\begin{proof}
Let $(\mathfrak{g},[\cdot,\cdot],\alpha,\varepsilon)$ be a color Hom-Lie algebra and  $\beta:\mathfrak{g}\rightarrow\mathfrak{g}$ be a color Hom-Lie algebras weak morphism. Consider a  bracket  $[\cdot,\cdot]_\beta$ on $\g$ defined by  $[\cdot,\cdot]_\beta=\beta\circ[\cdot,\cdot]$. Then, for any homogeneous elements $x,y,z\in \mathfrak{g}$ and using the $\varepsilon$-Hom-Jacobi identity \eqref{JacobiId1}, we have
\begin{align*}\circlearrowleft_{x,y,z}\varepsilon(z,x)[\beta\circ\alpha(x),[y,z]_\beta]_\beta
&=\circlearrowleft_{x,y,z}\varepsilon(z,x)\beta\circ[\beta\circ\alpha(x),\beta\circ[y,z]]\\
&=\circlearrowleft_{x,y,z}\varepsilon(z,x)\beta^2\circ[\alpha(x),[y,z]]\\
&=\beta^2(\circlearrowleft_{x,y,z}\varepsilon(z,x)[\alpha(x),[y,z]])=0.
\end{align*}
In addition, the fact that $[\cdot,\cdot]_\beta$ is $\varepsilon$-skew-symmetric implies that  $(\mathfrak{g},[\cdot,\cdot]_\beta,\beta\circ\alpha,\varepsilon)$ is a color Hom-Lie algebra.\\ 
\end{proof}

As a particular case we obtain the following examples:
\begin{example}\
\begin{enumerate}
\item Let $(\mathfrak{g},[\cdot,\cdot],\varepsilon)$ be a color Lie algebra and $\alpha$ be a  color Lie algebra morphism, then $(\mathfrak{g},[\cdot,\cdot]_\alpha,\alpha,\varepsilon)$ is a multiplicative   color Hom-Lie algebra.
\item Let $(\mathfrak{g},[\cdot,\cdot],\alpha, \varepsilon)$ be a multiplicative color Hom-Lie algebra then $(\mathfrak{g},\alpha^{n}\circ[\cdot ,\cdot ],\alpha^{n+1}, \varepsilon)$ is a multiplicative color Hom-Lie algebra for any $n\geq 0$ .  
\end{enumerate}
\end{example}


In the following, we extend to color Hom-Lie algebras the study of quadratic Hom-Lie algebras introduced in \cite{benayadi2010hom}.\\   

\begin{definition}
Let $(\g, [\cdot,\cdot],\alpha, \varepsilon)$ be a color Hom-Lie  algebra and $B:\g\times\g\rightarrow\K$ be a bilinear form on $\g$.
\begin{enumerate}
\item[(i)] $B$ is said $\varepsilon$-symmetric if $B(x,y)=\varepsilon(x,y)B(y,x)$, (called also color-symmetry).
\item[(ii)] $B$ is said invariant if $B([x,y],z)=B(x,[y,z])$, \ \ $\forall x,y,z\in\g$.\\
\end{enumerate}
\end{definition}

\begin{definition}
 A color Hom-Lie  algebra $(\g,[\cdot,\cdot],\alpha, \varepsilon)$ is called \textit{quadratic} color Hom-Lie  algebra if there exists a non-degenerate, $\varepsilon$-symmetric and invariant bilinear form $B$ on $\g$ such that $\alpha$ is $B$-symmetric (i.e. $B(\alpha(x),y)=B(x, \alpha(y))$). It is denoted by $(\g,[\cdot,\cdot],\alpha, \varepsilon,B)$ and $B$ is called invariant scalar product.\\
\end{definition}

\begin{remark}
We recover  quadratic color Lie  algebras when $\alpha=id_{\g}$ and quadratic Hom-Lie superalgebras when $\Gamma = {\mathbb Z}_2$ and $\varepsilon(\bar{1},\bar{1}) = -1$.\\
\end{remark}

\begin{definition}
A  color Hom-Lie algebra $(\g, [\cdot,\cdot], \alpha, \varepsilon)$ is called \emph{Hom-quadratic} if there exists $(B,\beta)$, where $B$ is a non-degenerate and $\varepsilon$-symmetric bilinear form on $\g$ and $\beta:\g \longrightarrow \g$ is an even homomorphism such that
$$B(\alpha(x),y)= B(x, \alpha(y))\ \ \ \ \mbox{and}\ \ \ \ B(\beta(x), \left[y,z\right])=B(\left[x,y\right],\beta(z)), \ \ \forall\ x,y,z\in \g.$$
We call the second condition the $\beta$-invariance of $B$.\\
\end{definition}

\begin{remark}
We recover  quadratic color Hom-Lie algebras when  $\beta=id_{\g}$.\\
\end{remark}

In the following, we give a similar result to Theorem \ref{ConsWeakMorph}, in the case of quadratic color Hom-Lie algebras. To this end, we consider $(\g,[\cdot,\cdot],\alpha, \varepsilon, B)$ a quadratic color Hom-Lie algebra and define:\\

$\bullet$ $Aut_{s}(\g,B)$ as the set consisting of  all automorphisms $\varphi$ of $\g$ which satisfy $B(\varphi(x),y)=B(x,\varphi(y)), \ \forall x,y\in \g$. We call $Aut_{s}(\g,B)$ the set of symmetric automorphisms of $(\g,[\cdot,\cdot],\alpha, \varepsilon,B)$.\\

$\bullet$ The subcategory of quadratic color Hom-Lie algebras, in which the objets are quadratic color Hom-Lie algebras and the morphisms are  isometry morphisms of quadratic color Hom-Lie algebras. A morphism $f$ between two  quadratic color Hom-Lie algebras $ ({\g}_1, B_1)$ and $({\g}_2,B_2)$ is said to be an isometry if  $f: {\g}_1 \rightarrow {\g}_2 $ satisfies $B_1(x,y) = B_2(f(x),f(y)), \ \forall x,y\in {\g}_1$.\\


\begin{thm}\label{ConsWeakMorphQu}
The subcategory of quadratic color Hom-Lie algebras is closed under symmetric automorphisms.\\
\end{thm}


\begin{proof}
Let $(\mathfrak{g},[\cdot,\cdot],\alpha,\varepsilon,B)$ be a quadratic color Hom-Lie algebra and $\beta$ be an element of $Aut_{s}(\g,B)$. Consider the bilinear form $B_{\beta}$ defined on $\g$ by $B_{\beta}(x,y) = B(\beta(x),y)$, $\forall x,y\in \g$. Using the fact that $\beta$ is $B$-symmetric  and $ \beta\circ\alpha=\alpha\circ\beta$, we obtain 
  $$B_{\beta}(\beta\circ\alpha(x),y)=B(\beta\circ \beta\circ\alpha(x),y)=B(\alpha\circ \beta(x),\beta(y))=B(\beta(x),\beta\circ\alpha(y))= B_{\beta}(x,\beta\circ\alpha(y))$$
and which means that $\beta\circ\alpha$ is $B_{\beta}$-symmetric. In addition
\begin{eqnarray*}
B_{\beta}([x,y]_\beta,z)=B(\beta\circ\beta [x,y],z)=B([\beta(x),\beta(y)],\beta(z))=B(\beta(x),[\beta(y),\beta(z)])=B_{\beta}(x,[y,z]_\beta), 
\end{eqnarray*}
 which means that $B_{\beta}$ is invariant with respect to ${\left[\cdot,\cdot\right]}_{\beta}$. The color symmetry and the non-degeneracy of $B_{\beta}$ follow from the symmetry of $B$ and the fact that  $\beta$ is bijective. Consequently, $(\mathfrak{g},[\cdot,\cdot]_\beta,\beta\circ\alpha,\varepsilon,B_{\beta})$ is a quadratic color Hom-Lie algebra.\\
\end{proof}

\begin{cor}\label{twistcentral2Qu}
Let $(\mathfrak{g},[\cdot,\cdot],\alpha, \varepsilon)$ be a multiplicative color Hom-Lie algebra.
If $B$ is an invariant scalar product on $\mathfrak{g}$, then  $(\mathfrak{g},\alpha^{n}\circ[\cdot ,...,\cdot ],\alpha^{n+1},B_{\alpha^{n}})$ is a Hom-quadratic multiplicative color Hom-Lie algebra and $(\mathfrak{g},\alpha^{n}\circ[\cdot ,...,\cdot ],\alpha^{n+1},B_\alpha)$, where $B_\alpha(x,y)=B(\alpha^n(x),y)=B(x,\alpha^n(y))$, is a quadratic multiplicative color Hom-Lie algebra.\\
\end{cor}



\section{Constructions and examples of quadratic color Hom-Lie  algebras} 


In this section, we describe in one hand some constructions of quadratic color Hom-Lie algebras and in the other hand they are used to provide examples.\\


Let $(\mathfrak{g},[\cdot,\cdot],\varepsilon)$ be a color Lie algebra and $End(\g)$ be the set of all homogeneous self linear maps on the  vector space $\g$. Clearly $End(\g)$ is $\Gamma$-graded and provided with the color-commutator introduced in \cite{yuan2010hom} and defined  by 
$${\left[f,g\right]}_{\varepsilon}= f\circ g - \varepsilon(f,g)\ g\circ f, \ \ \ \forall f\in {(End(\g))}_{f}, \ g\in {(End(\g))}_{g}$$
is a color Lie algebra.\\
 
\begin{definition}
The vector subspace $Cent(\g)$ of $End(\g)$ defined by
\begin{equation*}\label{Centroidalg}
Cent(\g)=\{\theta\in End(\mathfrak{g}):\ \theta ([x,y])=[\theta(x),y]=\varepsilon(\theta,x)[x,\theta(y)],\ \forall x,y\in\g\}.
\end{equation*}
 is a subalgebra of $End(\g)$ which we call the centroid of $\g$.\\ 
 
\end{definition}

%

In the following corollary, we  construct  color Hom-Lie algebras starting from a color Hom-Lie algebra and an even element in its centroid. This result is a graded version of a result obtained in \cite{benayadi2010hom}.\\

\begin{proposition}\label{constrcentr}
Let $(\mathfrak{g},[\cdot,\cdot], \alpha,\varepsilon)$ be a color Hom-Lie algebra and $\theta$ be an even element in the centroid of $\g$. Then we have the following  color Hom-Lie algebras
\begin{enumerate}
\item[(i)] $(\mathfrak{g},[\cdot,\cdot], \theta\circ\alpha,\varepsilon), \ \ \ \  (\g,[\cdot,\cdot]_1^{\theta},\theta\circ \alpha, \varepsilon),\ \ \ \ (\g,[\cdot,\cdot]_2^{\theta},\theta\circ \alpha, \varepsilon)$,
\item[(ii)] $(\mathfrak{g},[\cdot,\cdot],\alpha\circ \theta,\varepsilon),\ \ \ \ (\g,[\cdot,\cdot]_1^{\theta},\alpha\circ \theta, \varepsilon),\ \ \ \ (\g,[\cdot,\cdot]_2^{\theta},\alpha\circ \theta, \varepsilon)$,
\end{enumerate}
where the two even brackets $[\cdot,\cdot]_1^{\theta}$ and $[\cdot,\cdot]_2^{\theta}$ on $\g$ are defined by
$$[x,y]_1^{\theta}=[\theta(x),y]\ \ \ \ \mbox{and}\ \ \ [x,y]_2^{\theta}=[\theta(x),\theta(y)]\ \ \ \forall x,y\in \g.$$\\
\end{proposition}

\begin{proof}
$(i)$ Let $x$, $y$ and $z$ be  homogeneous elements in $\g$. The fact that $\theta$ is an even element of the centroid of $\g$ implies
$$\varepsilon(z,x) \big[ \theta\circ \alpha(x), [y,z]\big] = \theta \big( \varepsilon(z,x) \big[ \alpha(x), [y,z]\big]\big).$$
Consequently,  
$$\circlearrowleft_{x,y,z} \varepsilon(z,x) \big[ \theta\circ \alpha(x), [y,z]\big] = \theta \big(\circlearrowleft_{x,y,z} \varepsilon(z,x) \big[ \alpha(x), [y,z]\big] \big) = 0.$$
So, we deduce that $(\mathfrak{g},[\cdot,\cdot], \theta\circ\alpha,\varepsilon)$ is a color Hom-Lie algebra. To prove that $(\g,[\cdot,\cdot]_i^{\theta},\theta\circ \alpha, \varepsilon)$, $i=1,2$, are color Hom-Lie algebras, we need only to check that $(\g,[\cdot,\cdot]_i^{\theta},\alpha, \varepsilon)$, $i=1,2$, are color Hom-Lie algebras. Using again the fact that $\theta$ an even element in the centroid of $\g$, it comes 
$$\circlearrowleft_{x,y,z}\varepsilon(z,x)[\alpha(x),[y,z]_i^{\theta}]_i^{\theta} = \circlearrowleft_{x,y,z} [\theta\circ\alpha (x),[\theta(y),z]]=\theta^2\big(\circlearrowleft_{x,y,z}\varepsilon(z,x)[\alpha(x),[y,z]]\big)=0, \ \  i=1,2.$$
$(ii)$  Applying the $\varepsilon$-Hom-Jacobi identity to $\theta(x)$, $y$ and $z$ in $\g$, it follows 
$$\varepsilon(z,x) \big[\alpha\circ\theta(x), [y,z]\big] = - \theta \big( \varepsilon(x,y)\big[\alpha(y), [z,x]\big] + \varepsilon(y,z) \big[\alpha(z), [x,y]\big]\big).$$ 
Consequently
\begin{eqnarray*}
\circlearrowleft_{x,y,z}\varepsilon(z,x) \big[\alpha\circ\theta(x), [y,z]\big] = &-& \theta \big(\  \varepsilon(x,y)\big[\alpha(y), [z,x]\big] + \varepsilon(y,z) \big[\alpha(z), [x,y]\big]\ \big)\\
&-& \theta \big(\ \varepsilon(y,z)\big[\alpha(z), [x,y]\big] + \varepsilon(z,x) \big[\alpha(x), [y,z]\big]\ \big)\\
&-& \theta \big(\ \varepsilon(z,x)\big[\alpha(x), [y,z]\big] + \varepsilon(x,y) \big[\alpha(y), [z,x]\big]\ \big).
\end{eqnarray*}
Hence, we deduce that $(\g, [\cdot,\cdot], \alpha\circ \theta, \varepsilon)$ is a color Hom-Lie algebra. Since $(\g,[\cdot,\cdot]_i^{\theta},\alpha, \varepsilon)$, $i=1,2$, are color Hom-Lie algebra, reasoning similarly as above  proves the result.\\

\end{proof}

As an application of  Proposition \ref{constrcentr}, we construct a color Hom-Lie algebra such that the  twisted map is an element of its centroid. We assume that $(\g,[\cdot ,\cdot ],\alpha,\varepsilon)$ is a color Hom-Lie algebra and define the following vector space $\frak U$ of $End(\g)$ consisting of  even linear maps $\sigma$ on $\frak U$ as follows: 
$$\frak U =\{u\in End(\g), \ \ u\circ\alpha=\alpha\circ u\}$$ 
and 
$$\sigma: {\frak U} \longrightarrow {\frak U}\ \ ; \ \ \sigma(u)= \alpha\circ u.$$
Then, ${\frak U}$ is a color subalgebra of $End(\g)$. Moreover, $({\frak U}, {\left\{\cdot,\cdot\right\}}_1^{\sigma},\sigma)$ and $({\frak U}, {\left\{\cdot,\cdot\right\}}_2^{\sigma},\sigma)$ are two color Hom-Lie algebras. Indeed, by a simple computation, it comes that ${\frak U}$ is a color subalgebra of $End(\g)$. In addition, for two homogeneous elements $u$ and $v$ in ${\frak U}$, we have 
\begin{eqnarray*}
\sigma([u,v])=\alpha( u\circ v-\varepsilon(u,v) v\circ u )=(\alpha\circ u)\circ v-\varepsilon(u,v) v\circ(\alpha\circ u)=[\sigma(u),v]. 
\end{eqnarray*}
Now applying Proposition \ref{constrcentr}, we conclude that $({\frak U}, [\cdot,\cdot ]_1^{\sigma},\sigma)$ and $({\frak U}, [\cdot,\cdot]_2^{\sigma},\sigma)$ are two color Hom-Lie algebras.\\

We consider now constructions  using  elements of  the centroid in  the quadratic case.\\

\begin{proposition}
Let $(\g,\left[\cdot,\cdot\right],\varepsilon,B)$ be a quadratic color Lie algebra and $\theta$ be an invertible and $B$-symmetric element in the centroid of $\g$. Then, the two color Hom-Lie algebras   $(\g,{[\cdot,\cdot]}_1^{\theta},\theta,\varepsilon)$ and $(\g,{[\cdot,\cdot]}_2^{\theta},\theta,\varepsilon)$ defined in Proposition \ref{constrcentr} provided with the even bilinear form $B_{\theta}$ such that $B_{\theta}(x,y)=B(\theta(x),y)$,$\forall x,y\in \g$ are two quadratic color Hom-Lie algebras.\\ 
\end{proposition}


\begin{proof}Since $B$ is non degenerate color symmetric bilinear form and $\theta$ invertible then $\beta_\theta$ is a nondegenerate color symmetric bilinear form on $\g$. Moreover  for  homogeneous elements  $x,y,z$ in $\g$ we have
\begin{align*}
B_\theta({[x,y]}_1^{\theta},z)&=B(\theta[\theta (x),y],z) = B([\theta (x),y],\theta (z))=B(\theta (x),[y,\theta (z)])=B_\theta( x,[\theta (y), z])=B_\theta( x,{[y, z]}_1^{\theta}),\\
\end{align*}
and 
\begin{align*}
B_\theta({[x,y]}_2^{\theta},z)&=B(\theta[\theta (x),\theta(y)],z) = B([\theta (x),\theta(y)], \theta(z))=B(\theta (x),[\theta(y),\theta (z)])=B_\theta( x,{[y,z]}_2^{\theta}).
\end{align*}
\end{proof}


\begin{definition}
A \textit{color Hom-associative } algebra is a tuple $(A,\mu,\alpha,\varepsilon)$ consisting of a $\Gamma$-graded vector space $A$, an even bilinear map $\mu:A\times A\rightarrow A$ $($i.e. $\mu(A_a,A_b)\subset A_{a+b})$ and an even homomorphism $\alpha:A\rightarrow A$ such that for homogeneous elements $x,y,z\in A$ we have 
\begin{align*}
\mu(\alpha(x),\mu(y,z))=\mu(\mu(x,y),\alpha(z)).\\
\end{align*}
In the case where $\mu(x,y)=\varepsilon(x,y)\mu(y,x)$,  the color Hom-associative  algebra  $(A,\mu,\alpha)$ is called  commutative.\\
\end{definition}

\begin{definition}
A color Hom-associative algebra $(A,\mu,\alpha,\varepsilon)$ is called  \emph{quadratic}  if there exists a  color-symmetric, invariant and nondegenerate  bilinear form $B$ on $\g$ such that $\alpha$ is $B$-symmetric.\\
\end{definition}

\begin{thm}\label{constrHomcolor-Ass}
Given a quadratic color Hom-associative algebra $(A,\mu, \alpha,\varepsilon, B)$, one can define  color commutator on homogeneous elements $x,y\in A$ by  $$[x,y]=\mu(x,y)-\varepsilon(x,y)\mu(y,x),$$
and then extended by linearity to all elements. The tuple  $(A,[\cdot,\cdot],\alpha,\varepsilon,B)$ determines a quadratic color Hom-Lie algebra.\\
\end{thm}

\begin{proof}The fact that $(A,[\cdot,\cdot],\alpha,\varepsilon)$ is a color Hom-Lie algebra was proved in  \cite[Proposition 3.13]{yuan2010hom}. By using the fact that  $\varepsilon$ is bicharacter, we have for homogeneous elements $x,y,z$ in $A$
\begin{align*}
B([x,y],z)&=B(\mu(x,y)-\varepsilon(x,y)\mu(y,x),z)\\
&=B(\mu(x,y),z)-\varepsilon(x,y)B(\mu(y,x),z)\\
&=B(x,\mu(y,z))-\varepsilon(x,y)\varepsilon(x+y,z)B(z,\mu(y,x))\\
&=B(x,\mu(y,z))-\varepsilon(x,y)\varepsilon(x+y,z)B(\mu(z,y),x)\\
&=B(x,\mu(y,z))-\varepsilon(x,y)\varepsilon(x+y,z)\varepsilon(z+y,x)B(x,\mu(z,y))\\
&=B(x,\mu(y,z))-\varepsilon(y,z)B(x,\mu(z,y))\\
&=B(x,\mu(y,z))-B(x,\varepsilon(y,z)\mu(z,y))=B(x,[y,z]).
\end{align*}
Thus $B$ is invariant, hence  $(A,[\cdot,\cdot],\alpha,\varepsilon,B)$ is a quadratic color Hom-Lie algebra.\\
\end{proof}

Recall that if  $V$ and $V'$ are two $\Gamma$-graded vector spaces, then the tensor product $V\otimes V'$ is still a $\Gamma$-graded vector space such that for $\delta\in \Gamma$ we have
$${(V\otimes V')}_{\delta} = \sum_{\delta=\gamma + \gamma'} V_{\gamma}\otimes V_{\gamma'},$$
where $\gamma, \gamma'\in \Gamma$.


\begin{thm}\label{tensor-product}
Let $(A,\mu,\alpha_A)$ be a  commutative color Hom-associative algebra and $(\mathfrak{g}, {[\cdot,\cdot]}_{\mathfrak g}, \alpha_{\mathfrak g}, \varepsilon)$ be a color Hom-Lie algebra.
\begin{enumerate}
\item [(i)] The tensor product $(\g\otimes A, [\cdot,\cdot], \alpha, \varepsilon)$ is a color Hom-Lie algebra such that
\begin{eqnarray*}
\alpha(x\otimes a)&=&\alpha_{\g}(x)\otimes\alpha_A(a),\\
\left[ x\otimes a, y\otimes b \right]&=&\varepsilon(a,y)\ { \left[x,y\right]}_{\mathfrak g}\otimes \mu(a,b),\ \ \ \forall x,\, y\in {\mathfrak g}, \ a,b\in A.
\end{eqnarray*}
\item[(ii)] If $B_A$ and $B_{\g}$ are respectively associative scalar product on $A$ and invariant scalar product on $\g$, then $(\g\otimes A, [\cdot,\cdot], \alpha, \varepsilon)$ is a quadratic color Hom-Lie algebra such that the invariant scalar product $B$ on $\g\otimes A$ is given by
$$B(x\otimes a, y\otimes b):= \varepsilon(a,y)\  B_{\frak g}(x,y)B_A(a,b)\ \ \ \forall x,\, y\in {\mathfrak g}, \ a,b\in A.$$
\end{enumerate}
\end{thm}
\medskip


\begin{proof}
By using the definition of bracket and the commutativity of the product on the color Hom-associative algebra $ A$, we have

$$\varepsilon(z+c,x+a)\big[\alpha(x\otimes a), [y\otimes b, z\otimes c]\big] =  \varepsilon(z+c,x+a)\varepsilon(b,z)\varepsilon(a,y+z)\big[\alpha_{\g}(x),[y,z]\big]\otimes \mu(\alpha_A(a),\mu(b,c)).$$

$$\varepsilon(x+a,y+b)\big[\alpha(y\otimes b), [z\otimes c, x\otimes a]\big] = \varepsilon(x+a,y+b)\varepsilon(c,x)\varepsilon(b,z+x)\varepsilon(b+c,a) \big[\alpha_{\g}(y),[z,x]\big]\otimes \mu(\alpha_A(a),\mu(b,c)).$$

$$\varepsilon(y+b,z+c)\big[\alpha(z\otimes c), [x\otimes a, y\otimes b]\big] =  \varepsilon(y+b,z+c)\varepsilon(a,y)\varepsilon(c,x+y)\varepsilon(c,a+b) \big[\alpha_{\g}(z),[x,y]\big]\otimes \mu(\alpha_A(a),\mu(b,c)).$$
\medskip

Hence,  using the identities \eqref{condi1-bicharacter}, \eqref{condi2-bicharacter} and \eqref{condi3-bicharacter}, we prove  the $\varepsilon$-Hom-Jacobi identity. Moreover, we have
\begin{eqnarray*}
B(\big[x\otimes a, y\otimes b\big], z\otimes c) &=& \varepsilon(a,y)\varepsilon(a+b,z) B_{\frak g}([x,y],z)B_A(\mu(a,b),c)\\
&=& \varepsilon(a,y+z)\varepsilon(b,z) B_{\frak g}(x,[y,z])B_A(a,\mu(b,c))\\
&=& \varepsilon(b,z) B(x\otimes a, [y,z]\otimes \mu(b,c)) = B(x\otimes a, [y\otimes b, z\otimes c]).
\end{eqnarray*}
In addition,
\begin{eqnarray*}
B(\alpha(x\otimes a), y\otimes b) &=& \varepsilon(a,y) B_{\frak g}(\alpha_{\g}(x),y)B_A(\alpha_A(a),b)\\
&=& \varepsilon(a,y) B_{\frak g}(x,\alpha_{\g}(y))B_A(a,\alpha_A(b)) = B(x\otimes a, \alpha(y\otimes b)).\\
\end{eqnarray*}

Consequently, we obtain $(\g\otimes A, \alpha, [\cdot,\cdot],\varepsilon, B)$ is a quadratic color Hom-Lie algebra.\\
\end{proof}

\subsection{Examples of quadratic Hom-Lie superalgebras}
Color Hom-Lie  algebras have been studied in \cite{yuan2010hom} where many examples  are provided. In the following, we give some examples of quadratic color Hom-Lie algebras by using  methods of construction mentioned above.\\

\begin{example}\label{Hom-sl2}
The twisting maps which make $\mathfrak{sl}_2$ a  Hom-Lie algebra were given  in \cite{makhlouf2006hom}. Then   ${\frak g}$ defined, with respect to a basis $\left\{x_1,x_2,x_3\right\}$,  by
$$[x_1,x_2] =2x_2 , \ \ \ [x_1,x_3] =-2x_3 ,\ \ \ [x_2,x_3] = x_1,$$
and linear maps $\alpha$ defined, with respect to the previous basis,  by:
 $$ \left(
\begin{array}{ccc}
a \ & d & c \\
2c & b & f \\
2d & e & b \\
\end{array}
\right)$$
are Hom-Lie algebras.

In the case where the matrix is the identity matrix, one gets the classical Lie algebra $\mathfrak{sl}_2$. It is well known that the Lie algebra $\mathfrak{sl}_2$ provided with its Killing form ${\frak K}$ is a quadratic Lie algebra. Moreover, a straightforward computation shows that  $\alpha$ are ${\frak K}$-symmetric (i.e. ${\frak K}(\alpha(x),y)= {\frak K}(x,\alpha(y))$). Consequently,
 $({\frak g},[\cdot,\cdot ],\alpha,{\frak K})$  are quadratic Hom-Lie algebras.\\
\end{example}

\begin{example}\label{Hom-nilpotent}
Nilpotent Lie superalgebras up to dimension 5 are classified in \cite{hegazi1999classification}. Among these Lie superalgebras, we find a class of quadratic 2-nilpotent Lie superalgebras   ${\frak L}  =  {\frak  L}_{\overline{0}}\oplus {\frak L}_{\overline{1}}$ where ${\frak  L}_{\overline{0}}$ is generated by $  <l_{0}, k_0>$ and ${\frak L}_{\overline{1}}$ is generated by  $ <l_1 , k_1>$. The bracket  being defined as $\left[l_0 , l_1\right] = k_1$, $\left[l_1 , l_1\right] = k_0$.  The quadratic structures are given by bilinear forms $B_{p,q}$ which are defined with respect to  the basis $\left\{l_0, k_0, l_1, k_1\right\}$ by the matrices 
$$ \left(
\begin{array}{cccc}
p& -q & 0 & 0 \\
-q & 0 & 0 & 0 \\
0 & 0 & 0 & q \\
0 & 0 & -q & 0
\end{array}
\right)$$ where $p\in {\mathbb K} $ and $q \in {\mathbb K}\setminus \left\{0\right\}$. 

Now, let us consider the even linear map $\alpha:  {\frak L}\longrightarrow {\frak L}$  defined, with respect to the basis $\left\{l_0,k_0,l_1,k_1\right\}$, by the matrices:
$$\left(
\begin{array}{llll}
a & \ \ \ c & \ \ 0 & \ \ 0\\
b & a + \frac{p}{q}c &\ \  0 & \ \ 0\\
0 &\ \  \ 0 & \ \ d & \ \ 0\\
0 & \ \ \ 0 & \ \ 0 & \ \ d
\end{array}
\right).$$
Direct  calculations show that super-Hom-Jacobi identity is satisfied and  
\begin{eqnarray*}
B(\alpha(l_0),k_0) &=& B(a l_0 + b k_0, k_0)= -aq = cp -q(a + \frac{p}{q}c)= B(l_0, cl_0 + (a + \frac{p}{q}c)k_0) = B(l_0, \alpha(k_0)),
\end{eqnarray*}
\begin{eqnarray*}
B(\alpha(l_1),l_1)=B(l_1, \alpha(l_1))= d\ B(l_1,l_1)= 0 = d\ B(k_1,k_1)= B(\alpha(k_1),k_1)=B(k_1, \alpha(k_1)),
\end{eqnarray*}
\begin{eqnarray*}
B(\alpha(l_1),k_1)= B(dl_1,k_1) = dq = B(l_1, dk_1) = B(l_1, \alpha(k_1)).
\end{eqnarray*}
So, we deduce that $({\frak L},[\cdot , \cdot ], \alpha, B)$ are quadratic Hom-Lie superalgebras.\\
\end{example}
\begin{example}
The two-dimensional superalgebra $A=A_{\bar{1}}$ is a symmetric associative super-commutative superalgebra. Applying  Theorem \ref{tensor-product}, where $\Gamma={\mathbb Z}_2$, to previous examples\\

$\bullet$ Let  $({\mathfrak g},[\cdot ,\cdot ], \alpha, {\frak K})$ be a quadratic Lie superalgebra defined  in Example \ref{Hom-sl2}. then   ${\mathfrak g}\otimes A$ is a quadratic Hom-Lie superalgebra where the twist map is given by $\alpha\otimes id$.\\

$\bullet$  Let  $({\frak L},[\cdot ,\cdot ],\alpha, B)$ be a quadratic Lie superalgebra defined in Example \ref{Hom-nilpotent}. then  ${\frak L}\otimes A$ is a quadratic Hom-Lie superalgebra where the twist map is given by $\alpha\otimes id$.\\
\end{example}

\begin{example}
Let us consider the class of symmetric associative super-commutative superalgebra defined by $A= A_{\bar{0}}\oplus A_{\bar{1}}$, where $dim A_{\bar{0}}= dim A_{\bar{1}}=2$, such that $A_{\bar{0}}. A_{\bar{1}}=\left\{0\right\},\ \  A_{\bar{1}}.A_{\bar{1}}=\left\{0\right\}$
and if we assume that $A_{\bar{0}}={\mathbb K}e_0\oplus {\mathbb K}f_0$, then we have $e_0\in Ann(A_{\bar{0}})$ and $f_0\cdot f_0= a e_0$, where $a \neq 0$.
In addition, the symmetric structures are  expressed, with respect to   the basis $\left\{e_0, f_0, e_1, f_1\right\}$,  by the following matrices
$$ \left(
\begin{array}{cccc}
0 & \alpha & 0 & 0 \\
\alpha & \beta & 0 & 0 \\
0 & 0 & 0 & \gamma \\
0 & 0 & -\gamma & 0
\end{array}\right)$$ where $\alpha, \gamma \in {\mathbb K}\setminus \left\{0\right\}$ and $\beta\in {\mathbb K}$.\\

 Theorem \ref{tensor-product}, where $\Gamma={\mathbb Z}_2$, leads to ${\mathfrak g}\otimes A$ (resp. ${\frak L}\otimes A$), where $\g$ is defined in Example \ref{Hom-sl2}  (resp. ${\frak L}$ defined in Example \ref{Hom-nilpotent}) is a quadratic Hom-Lie superalgebra where the twist map is given by $\alpha\otimes id$.

\end{example}

\begin{example}
Let $A$ be an $n$-dimensional  vector space  and $\wedge A $ be a Grassmann algebra of $A$. We know that $\wedge A$ is an super-commutative associative  superalgebra with
$${(\wedge A)}_{\bar{0}}= {\bigoplus}_{i\in {\mathbb Z}} {\wedge}^{2i} A \ \ \ \mbox{and}\ \ \ {(\wedge A)}_{\bar{1}}= {\bigoplus}_{i\in {\mathbb Z}} {\wedge}^{2i+1} A.$$
 Moreover, it has been proved in \cite{benayadi2003socle} that $\wedge A$ is a symmetric superalgebra if and only if the dimension $n$ of $A$ is even. Here we consider  $\wedge A$ where $A$ is a vector space with even dimension. We denote by $\theta$ an associative symmetric structure on $\wedge A$.\\
Similarly Theorem \ref{tensor-product}, where $\Gamma={\mathbb Z}_2$, leads to ${\mathfrak g}\otimes A$ (resp. ${\frak L}\otimes A$), where $\g$ is defined in Example \ref{Hom-sl2}  (resp. ${\frak L}$ defined in Example \ref{Hom-nilpotent}) is a quadratic Hom-Lie superalgebra where the twist map is given by $\alpha\otimes id$.

\end{example}


\section{Representation and some extensions of color Hom-Lie algebras}

\subsection{Representation of color Hom-Lie algebras and semi-direct product}
In this section we extend representation theory of Hom-Lie  algebras introduced in \cite{sheng2010representations} and \cite{benayadi2010hom} to  color Hom-Lie case.

\begin{definition}
Let $(\g, {\left[\cdot,\cdot\right]}_{\g}, \alpha, \varepsilon)$ be a color Hom-Lie algebra and  $(M,\beta)$ be a  pair consisting of   $\Gamma$-graded vector space $M$ and an even homomorphism  $\beta :M\rightarrow M$. The pair $(M,\beta)$ is said to be a  \emph{Hom-module} over $\g$ (or $\g$-Hom-module) if there exists a color skew-symmetric bilinear map $\left[\cdot,\cdot\right]:\g\times M \longrightarrow M$, that is $\left[x,m\right] = -\varepsilon(x,m)\left[m,x\right]$, $\forall x\in \g$ and $\forall m\in M$, such that
$$\varepsilon(m,x)\big[\alpha(x), \left[y,m\right]\big]+ \varepsilon(m,x)\big[\alpha(y), \left[m,x\right]\big] + \varepsilon(y,m)\left[\beta(m), {\left[x,y\right]}_{\g}\right] =0, \ \ \ \forall\ x,y\in \g, \forall m\in M.$$
\end{definition}

\begin{definition}
Let $(\g,[\cdot,\cdot],\alpha, \varepsilon)$ be a color Hom-Lie algebra and $(M,\beta)$ be a pair consisting  of a $\Gamma$-graded vector space $M$ and an even homomorphism  $\beta: M\rightarrow M$. A \emph{representation} of $\g$ on $(M,\beta)$  is an even linear map $\rho:\g \rightarrow \mathfrak{gl}(M)$ satisfying
\begin{equation}\label{RepreIdent}
\rho([x,y])\circ\beta=\rho(\alpha(x))\circ\rho(y)-\varepsilon(x,y)\rho(\alpha(y))\circ\rho(x).\\
\end{equation}
\end{definition}

\begin{definition}
A representation  $\rho$ of a multiplicative color Hom-Lie algebra $(\g, \alpha, \varepsilon)$ on a $\Gamma$-graded vector space $(M,\beta)$ is a representation of color Hom-Lie algebra (i.e. a linear map satisfying (\ref{RepreIdent})) such that the  following condition holds:
\begin{eqnarray}\label{condition-two-of-multiplicative-representation}
\beta \big( \rho(x)(m)\big) = \rho(\alpha(x))(\beta(m)), \ \ \ \forall \ m\in M\ \ \mbox{and}\ \forall x\in \ \g.
\end{eqnarray}
\end{definition} 

Now, let $(\g,[\cdot,\cdot],\alpha, \varepsilon)$ be a color Hom-Lie algebra and $(M,\beta)$ be a pair  of  a $\Gamma$-graded vector space $M$ and an even homomorphism  $\beta$. If  $\rho$ is a representation of $\g$ on $(M,\beta)$, then we can see easily that $M$ is a $\g$-module via $\left[x,m\right]= \rho(x)(m)$, $\forall x\in \g$ and $\forall m\in M$. Conversely, if $M$ is a $\g$-Hom-module, then the linear map $\rho:\g \longrightarrow \mathfrak{gl}(M)$ defined by $\rho(x)(m)= \left[x,m\right]$, $\forall x\in \g$ and $\forall m\in M$, is a representation of $\g$ on $(M,\beta )$.\\

\medskip

Two representations $\rho$ and $\rho'$ of $\g$ on $(M,\beta)$ and $(M',\beta')$ respectively are \emph{equivalent} if there exists  an even isomorphism of vector spaces $f:M\rightarrow M'$ such that $f\circ \beta = \beta'\circ f$ and $ f\circ \rho(x)= \rho'(x)\circ f$, $\forall x\in \g$.\\

\begin{example}
Let $(\g,[\cdot,\cdot],\alpha, \varepsilon)$ be a color Hom-Lie  algebra. The even linear map
$$ad: \g \longrightarrow \mathfrak{gl}(\g)\ \ \ \mbox{defined by}\ \ \ \  ad(x)(y):= \left[x,y\right], \ \ \forall x,y\in \g$$ is a representation of $\g$ on $(\g, \alpha)$. This representation is called  adjoint representation.\\
\end{example}

\begin{proposition}
Let $(\g,[\cdot,\cdot],\alpha,\varepsilon)$ be a  color Hom-Lie  algebra and $\rho$ be a representation of $\g$ on $(M,\beta)$. Let us consider $M^*$ the dual space of $M$ and $\tilde{\beta}:M^* \longrightarrow M^*$ an even homomorphism defined by $\widetilde{\beta}(f) = f\circ \beta$, $\forall f\in M^*$. Then, the even linear map $\widetilde{\rho}:\g\rightarrow End(M^*)$ defined by  $\widetilde{\rho} := -^t\rho$, that is $\widetilde{\rho}(x)(f) = - \varepsilon(x,f) f\circ \rho(x)$, $\forall\, f\in M^*$ and $\forall x\in \g$, is a representation of $\g$ on $(M^*,\tilde{\beta})$ if and only if
\begin{equation}\label{dualrep}
\rho(x)\circ \rho(\alpha(y))-\varepsilon(x,y)\rho(y)\circ\rho(\alpha(x))=\beta\circ\rho([x,y]),\ \forall\ x,y\in\g.\\
\end{equation}

\end{proposition}

\begin{proof}
Let $f\in M^*$, $x,y\in\g$ and $m\in M$. We compute the right hand side of the identity \eqref{RepreIdent}, we have
\begin{align*}
& \big(\widetilde{\rho}(\alpha(x))\circ\widetilde{\rho}(y)-\varepsilon(x,y)\widetilde{\rho}(\alpha(y))\circ\widetilde{\rho}(x)\big)(f)(m)\\
& =-\varepsilon(x,y+f)\widetilde{\rho}(y)(f)\circ\rho(\alpha(x))(m)+\varepsilon(y,f)\widetilde{\rho}(x)(f)\circ\rho(\alpha(y))(m)\\
& = \varepsilon(y,f)\varepsilon(x,y+f)f(\rho(y)\circ\rho(\alpha(x))(m))-\varepsilon(x+y,f) f(\rho(x)\circ\rho(\alpha(y))(m))\\
& = -\varepsilon(x+y,f) f\big( - \varepsilon(x,y) \rho(y)\circ\rho(\alpha(x))(m) + \rho(x)\circ\rho(\alpha(y))(m)\big).\\
\end{align*}
On the other hand, we set that the twisted map $\widetilde{\beta}$ is $\widetilde{\beta}=^t\beta$, then the left hand side \eqref{RepreIdent}  writes
\begin{align*}
\widetilde{\rho}([x,y])\circ\widetilde{\beta}(f)(m)&=-\varepsilon(x+y,f)\widetilde{\beta}(f)\circ\rho([x,y])(m)\\
&=-\varepsilon(x+y,f) f \big(\beta\circ\rho([x,y])\big)(m).\\
\end{align*}
Thus \eqref{dualrep} is satisfied.
\end{proof}

\begin{cor}
Let $ad$ be the adjoint representation of a color Hom-Lie  algebra $(\g, [\cdot,\cdot], \alpha, \varepsilon)$ and let us consider the even linear map  $\pi:\g\rightarrow End(\g^*)$ defined by $,\ \pi(x)(f)(y)=- \varepsilon(x,f)(f\circ ad(x))(y)$, $\forall x,y\in \g$. Then $\pi$ is a representation of $\g$ on $(\g^*,\widetilde{\alpha})$  if and only if
 \begin{equation}\label{coadrep}
ad(x)\circ ad(\alpha(y))-\varepsilon(x,y)ad(y)\circ ad(\alpha(x))=\alpha\circ ad([x,y]),\ \forall\ x,y\in\g.\\
\end{equation}
We call the representation $\pi$ the \emph{coadjoint representation} of $\g$.\\
\end{cor}
\begin{proposition}
Let $(M,\rho,\gamma)$ be a representation of color Hom-Lie  algebra $(\g,[\cdot,\cdot]_\g,\alpha)$, then the space $\g\oplus M$ is a $\Gamma$-graded color Hom-Lie  algebra where  $(\g\oplus M)_a=\g_a\oplus M_a$, with structure $([\cdot,\cdot],\eta)$ given, for  $x,y\in\g$ and $u,v\in M$ by
\begin{align}
&[x+u,y+v]=[x,y]_\g+\rho(x)(v)-\varepsilon(x,y)\rho(y)(u),\\
&\eta(x+u)=\alpha(x)+\gamma(u),
\end{align}
we call the direct sum $\g\oplus M$ semidirect product  of $\g$ and $M$. It is denoted by $\g\ltimes M$.
\end{proposition}
\begin{proof}
Observe that, for any three homogeneous elements $x+u,y+v$ and $z+w$ of $\g\oplus M$
\begin{align*}
\varepsilon(z,x)[\eta(x+u),[y+v,z+w]]&=\varepsilon(z,x)[\alpha(x)+\gamma(u),[y,z]_\g+\rho(y)(w)-\varepsilon(y,z)\rho(z)(v)]\\
&=\varepsilon(z,x)[\alpha(x),[y,z]_\g]_\g +\rho(\alpha(x))\circ\rho(y)(w)\\
&-\varepsilon(z,x)\varepsilon(y,z)\rho(\alpha(x))\circ\rho(z)(v)-\varepsilon(x,y)\rho([y,z]_\g)(\gamma(u)).
\end{align*}
And  similarly
\begin{align*}
\varepsilon(x,y)[\eta(y+v),[z+w,x+u]]&=\varepsilon(x,y)[\alpha(y),[z,x]_\g]_\g +\rho(\alpha(y))\circ\rho(z)(u)\\
&-\varepsilon(x,y)\varepsilon(z,x)\rho(\alpha(y))\circ\rho(x)(w)-\varepsilon(y,z)\rho([z,x]_\g)(\gamma(v)).
\end{align*}\begin{align*}
\varepsilon(y,z)[\eta(z+w),[x+u,y+v]]&=\varepsilon(y,z)[\alpha(z),[x,y]_\g]_\g +\rho(\alpha(z))\circ\rho(x)(v)\\
&-\varepsilon(y,z)\varepsilon(x,y)\rho(\alpha(z))\circ\rho(y)(u)-\varepsilon(z,x)\rho([x,y]_\g)(\gamma(w)).
\end{align*}
Therefore the identities \eqref{JacobiId1}, \eqref{RepreIdent} induce that $(\g\oplus M,[\cdot,\cdot],\eta)$ is a color Hom-Lie algebra.
\end{proof}

\section{Extensions}

In this section, we generalize to the case of color Hom-Lie algebras the notion of central extension and $T^*$-extension introduced in \cite{wang2010non} for color Lie algebras. These two types of extensions require some elements of cohomology of color Hom-Lie algebras. 

\subsection{Elements of cohomology}
The cohomology of color Lie algebras was introduced in \cite{scheunert1998cohomology}. In the following, we define  the first and the second cohomology groups of color Hom-Lie algebras, with values in a $\g$-Hom-module $(M,\beta )$.\\ 

Let $(\g,[\cdot,\cdot],\alpha,\varepsilon)$ be a color Hom-Lie algebra and $(M,\beta)$ be a $\Gamma$-graded $\g$-Hom-module. 
we set  
\begin{eqnarray*}
\mathscr{C}^n(\g,M) &=& \left\{\varphi: {\wedge}^n\g \longrightarrow M, \ \ \ \mbox{such that}\ \ \ \varphi\circ \alpha^{\otimes n}=\beta\circ \varphi\right\}, \quad \mbox{for } n\geq 1.\\
\mathscr{C}^n(\g,M) &=& \left\{0\right\}, \quad \mbox{for}\ n\leq -1,\\
\mathscr{C}^0(\g,M) &=& M.
\end{eqnarray*}
A homogeneous element $\varphi$ in $\mathscr{C}^n(g,M)$ is called an $n$-\emph{cochain}. Next, we define the  coboundary operators ${\delta}^n: \mathscr{C}^n(\g,M) \longrightarrow \mathscr{C}^{n+1}(\g,M)$ for $n=1,2$ by 
 \begin{eqnarray*}
{\delta}^1: \mathscr{C}^1(\g, M) &\longrightarrow &  \mathscr{C}^2(\g, M)\\
\varphi & \longmapsto & {\delta}^1(\varphi)
\end{eqnarray*}  

such that, for  homogeneous elements $x_0,x_1\in \g$,
\begin{eqnarray*}  \delta^1\varphi(x_0,x_1)=\varepsilon(\varphi,x_0)\rho(x_0)(\varphi(x_1)))-\varepsilon(\varphi+x_0,x_1)\rho(x_1)(\varphi(x_0))-\varphi([x_0,x_1]).
\end{eqnarray*} 
and
\begin{eqnarray*}
{\delta}^2: \mathscr{C}^2(\g, M) &\longrightarrow &  \mathscr{C}^3(\g, M)\\
\varphi & \longmapsto & {\delta}^2(\varphi)
\end{eqnarray*}  

such that, for  homogeneous elements $x_0,x_1,x_2\in \g$,
\begin{align*}
\delta^2\varphi(x_0,x_1,x_2)&=\varepsilon(\varphi,x_0)\rho(\alpha(x_0))(\varphi(x_1,x_2)))-\varepsilon(\varphi+x_0,x_1)\rho(\alpha(x_1))(\varphi(x_0,x_2))\\
&+\varepsilon(\varphi+x_0+x_1,x_2)\rho(\alpha(x_2))(\varphi(x_0,x_1))-\varphi([x_0,x_1],\alpha(x_2))\nonumber\\
&+\varepsilon(x_1,x_2)\varphi([x_0,x_2],\alpha(x_1))+\varphi(\alpha(x_0),[x_1,x_2])\nonumber.\\
\end{align*}

\begin{proposition}\label{delta2-circ-delta1-equal-zero}
Let $(\g, [\cdot, \cdot],\alpha,\varepsilon)$ be a color Hom-Lie algebra and ${\delta}^n: \mathscr{C}^n(\g,M) \longrightarrow \mathscr{C}^{n+1}(\g,M)$ be the coboundary operator defined  above. Then, the composite  $\delta^2\circ\delta^1=0$.\\
\end{proposition}

\begin{remark}
In order to define a cohomology complex, we may set $({\delta}^0m)(x)=m\cdot x$, $\forall x\in \g$ and $\forall m\in M$, and ${\delta}^n = 0$, $\forall n \geq 3$. Deeply constructions will be given in a forthcoming paper.\\
\end{remark}

We denote the kernel of ${\delta}^n$ by $\mathscr{Z}^n(\g, M)$, whose  elements  are called $n$-\emph{cocycles},  and the image of ${\delta}^{n-1}$ by $\mathscr{B}^n(\g, M)$, whose elements are called $n$-\emph{coboundaries}. Moreover $\mathscr{Z}^n(\g, M)$ and $\mathscr{B}^n(\g, M)$ are two graded submodules of $\mathscr{C}^n(\g, M)$. Following  Proposition \ref{delta2-circ-delta1-equal-zero} we have  $$\mathscr{B}^n(\g, M)\subset \mathscr{Z}^n(\g, M).$$

Consequently, we construct the so-called \emph{cohomology groups} $$\mathscr{H}^n(\g, M)= \mathscr{Z}^n(\g,M)/\mathscr{B}^n(\g,M).$$
 Two elements of $\mathscr{Z}^n(\g,M)$ are said to \emph{cohomologous} if their residue classes modulo $\mathscr{B}^n(\g, M)$ coincide, that is if their difference lies in $\mathscr{B}^n(\g,M)$.\\ 



\subsection{Central extensions}

Let $({\g}_i, {[\cdot,\cdot]}_i, {\alpha}_i, \varepsilon)$, $i=1,2$, be two color Hom-Lie algebras. A color Hom-Lie algebra $\g$ is an extension of ${\g}_2$ by ${\g}_1$ if there exists an exact sequence 
$$0 \longrightarrow {\g}_1 \stackrel{{\mu}_1}{\longrightarrow} \g \stackrel{{\mu}_2}{\longrightarrow} {\g}_2 \longrightarrow 0$$ 
such that ${\mu}_1$ and ${\mu}_2$ are two morphisms of color Hom-Lie algebras. The kernel of ${\mu}_2$ is said to be the kernel of the extension. Two extensions $\g$ and $\g'$ of ${\g}_2$ by ${\g}_1$ are said to be equivalent if there exists  a morphism of color Hom-Lie algebras $f : \g \longrightarrow \g'$ such that the following diagram commutes.

$$\begin{tabular}{ccccccc}
$0 \longrightarrow$ & ${\g}_1$ & $\stackrel{{\mu}_1}{\longrightarrow}$ & $\g$ & $\stackrel{{\mu}_2}{\longrightarrow}$ & ${\g}_2$ & $\longrightarrow 0$\\
                  & $\downarrow$ &                                  & $\downarrow$ &                            & $\downarrow$ &    \\
$0 \longrightarrow$ & ${\g}_1$ & $\stackrel{{\mu'}_1}{\longrightarrow}$ & $\g'$ & $\stackrel{{\mu'}_2}{\longrightarrow}$ & ${\g}_2$ & $\longrightarrow 0$\\
\end{tabular}$$

\begin{definition}
An extension $\g$ of ${\g}_2$ by ${\g}_1$ is said to be central if the kernel of the extension lies in the center of $\g$.\\
\end{definition}

The following theorem gives a characterization of the bracket of the central extension of a color Hom-Lie algebra $\g$ and $\Gamma$-graded vector space $M$.\\

\begin{thm}
Let $(\g, {[\cdot,\cdot]}_{\g}, {\alpha}_{\g}, \varepsilon)$ be a color Hom-Lie algebra, $M$ be a $\Gamma$-graded vector space and $\Psi:\g\times \g \longrightarrow M$ be an even bilinear map. Then, the $\Gamma$-graded vector space $\g\oplus M$, where ${(\g\oplus M)}_{\gamma}= {\g}_{\gamma}\oplus M_{\gamma}$ for $\gamma\in \Gamma$, provided with the following bracket and   even linear map defined respectively by     
\begin{eqnarray}
\label{product-central-extension}[x + m, y + n] &=& {[x,y]}_{\g} + \Psi(x,y),\\
\label{twisting-central-extension}\alpha(x + m) &=& {\alpha}_{\g}(x) + m, \ \ \ \ \ \forall\ x,y\in \g, \ \ m,n\in M,
\end{eqnarray}
is a color Hom-Lie algebra central extension of $\g$ by $M$ if and only if$$\circlearrowleft_{x,y,z}\varepsilon(z,x)\Psi({\alpha}_{\g}(x), {[y,z]}_{\g})=0.$$  
In particular, if $\g$ is multiplicative and $\Psi$ satisfies $\Psi(\alpha(x),\alpha(y))= \Psi(x,y)$, $\forall x,y\in \g$, then the color Hom-Lie algebra $\g\oplus M$ is also multiplicative.\\
\end{thm}

\begin{proof}
We have for homogeneous elements $x,y,z\in\g$ and $m,n,l\in  M$,
\begin{eqnarray*}
\varepsilon(z,x) \big[\alpha_{\g}(x) + l, [y + m, z + n]\big] &=& \varepsilon(z,x) \big[{\alpha}_{\g}(x) + l, {[y,z]}_{\g} + \Psi(y,z)\big]\\
&=& \varepsilon(z,x) {\big[{\alpha}_{\g}(x), {[y,z]}_{\g}\big]}_{\g} + \varepsilon(z,x)\Psi({\alpha}_{\g}(x), {[y,z]}_{\g}).
\end{eqnarray*}
Consequently, 
$$\circlearrowleft_{x,y,z}\varepsilon(z,x) \big[\alpha_{\g}(x) + l, [y + m, z + n]\big] = 0$$
if and only if
$$\circlearrowleft_{x,y,z}\varepsilon(z,x)\Psi({\alpha}_{\g}(x), {[y,z]}_{\g})=0.$$
\end{proof}

\begin{cor}
Let $(\g, {[\cdot,\cdot]}_{\g}, {\alpha}_{\g}, \varepsilon)$ be a multiplicative color Hom-Lie algebra. Then, the $\Gamma$-graded vector space $\g\oplus \mathbb K$ provided with the product (\ref{product-central-extension}) and the map (\ref{twisting-central-extension}) is a multiplicative color Hom-Lie algebra if and only if $\Psi\in \mathscr{Z}^2(\g, \mathbb K)$.\\    
\end{cor}

Denote by $\mathcal{E}(\g,\mathbb K)$ the set of all equivalent classes of central extensions of a multiplicative color Hom-Lie algebra $\g$ by $\mathbb K$.\\

\begin{cor}\label{central-extension}
There exists a one-to-one correspondence between  $\mathcal{E}(\g,\mathbb K)$ and  the  color Hom-Lie second cohomology group $\mathscr{H}^2(\g, \mathbb K)$.\\
\end{cor}

Following \cite{sheng2010representations}, we define $\alpha^k$-derivations.

\begin{definition}
Let $(\g, [\cdot, \cdot], \alpha, \varepsilon)$ be a multiplicative color Hom-Lie algebra. A homogeneous bilinear map    $D: \g \longrightarrow \g$ of degree $d$ is said to be an $\alpha^k$-derivation, where $k\in \mathbb N$, if it satisfies 
$$D\circ \alpha = \alpha\circ D,$$
$$D[x,y] = [D(x),{\alpha}^k(y)] + \varepsilon(d,x)[{\alpha}^k(x), D(y)], $$
for any homogeneous element $x$ and for any $y\in \g$.
\end{definition}

The set of all $\alpha^k$-derivations is denoted by $Der_{\alpha}^k(\g)$.  The space $Der(\g) = \bigoplus_{k\geq 0} Der_{\alpha^k}(\g)$ provided with the color-commutator is a color  Lie algebra. The fact that $ Der_{\alpha^k}(\g)$ is $\Gamma$-graded implies that $Der(\g)$ is $\Gamma$-graded 
$${(Der(\g))}_{\gamma} = \bigoplus_{k\geq 0} {(Der_{\alpha^k}(\g))}_{\gamma}, \ \ \ \forall \ \gamma\in \Gamma.$$
Moreover, similarly to the non-graded case \cite{sheng2010representations}, equipped with the color commutator and the following even map  
$$\widetilde{\alpha}: Der(\g) \longrightarrow Der(\g)\ ;\ D \longmapsto \widetilde{\alpha}(D)=D\circ \alpha, $$
$Der(\g)$ is a color Hom-Lie algebra.\\

\begin{definition}
Let $(\g, [\cdot, \cdot], \alpha, \varepsilon, B)$ be a quadratic color Hom-Lie algebra and $D: \g \longrightarrow \g$ be a homogeneous derivation of degree $d$. The derivation $D$ is said to be $\varepsilon$-skew-symmetric with respect to $B$ (or $B$-skew-symmetric) if it satisfies 
$$B(D(x),y) = - \varepsilon(d, x)B(x, D(y)), \ \ \forall x\in {\g}_x, \forall y\in \g.$$
\end{definition}

In the following we  characterize  scalar cocycles of quadratic multiplicative color Hom-Lie algebras using  $\varepsilon$-skew-symmetric derivations.\\
  
\begin{lem}\label{lemma}
Let $(\g,[\cdot,\cdot], \alpha, B, \varepsilon)$ be a quadratic multiplicative color Hom-Lie algebra. 
\begin{enumerate}
\item[(i)] If $\omega$ is a homogeneous scalar $2$-cocycle, then there exists a homogeneous $\varepsilon$-skew-symmetric $\alpha$-derivation $D$ of $\g$ such that 
\begin{equation}\label{equivalent-2cocycle-derivation}
\omega(x,y):=B(D(x),y),\ \forall x,y\in\g
\end{equation}
\item[(ii)] Conversely, if $D$ a homogeneous $\varepsilon$-skew-symmetric $\alpha$-derivation of $\g$, then $\omega$ defined in (\ref{equivalent-2cocycle-derivation}) is a homogeneous scalar $2$-cocycle  on $\g$ that is 
$$\circlearrowleft_{x,y,z}\omega\big(\alpha(x), [y,z]\big)=0, \ \ \forall x,y,z\in \g.$$
\end{enumerate}
\end{lem}

\begin{proof}
$(i)$ The fact that $B$ is nondegenerate implies $\phi:\g \longrightarrow \g^*$ defined by $\phi(x)=B(x,.)$ is an isomorphism of $\Gamma$-graded vector spaces. Let $y\in \g$, since $\omega(y,.)\in \g^*$ then there exists a unique $Y_y\in \g$ such that $\phi(Y_y)=\omega(y,.)$, namely $B(Y_y,z) = \omega(y,z)$, $\forall z\in \g$. Set $D(y) = Y_y$, the map $D$ is well-defined by the uniqueness of $Y_y$. By using the fact that $B$ is even and $\omega$ is homogeneous, we deduce that $D$ is homogeneous with the same degree as $\omega$. The $B$-skew-symmetry of $D$ comes from the $\varepsilon$-symmetry of $B$ and the $\varepsilon$-skew-symmetry of $\omega$. Finally, for  homogeneous elements $x,y,z\in\g$, we have
$$\omega(\alpha(x),[y,z])=\omega([x,y],\alpha(z))+\varepsilon(x,y)\omega(\alpha(y),[x,z]),$$
  which is equivalent to
   $$B(D(\alpha(x)),[y,z])=B(D([x,y]),\alpha(z))+\varepsilon(x,y)B(D(\alpha(y)),[x,z]).$$
   Thus
   $$B([D(x),\alpha(y)],\alpha(z))=B(D([x,y]),\alpha(z))+\varepsilon(x,y)B([D(y),\alpha(x)],\alpha(z)).$$
Nondegeneracy of $B$  leads to
  $$D([x,y])=[D(x),\alpha(y)]-\varepsilon(x,y)[D(y),\alpha(x)].$$
$(ii)$ is obtained by straightforward computations.\\

\end{proof}

The following corollary which is an immediately consequence of  Theorem \ref{central-extension} and  Lemma \ref{lemma}, defines the first extension. It is used to define the double extension of quadratic color Hom-Lie algebras.\\

\begin{cor}
Let  $(\g, {[\cdot,\cdot]}_{\g}, {\alpha}_{\g},B)$ be a quadratic color Hom-Lie algebra and $D: \g\longrightarrow \g$ be an even skew-symmetric $\alpha$-derivation, then the $\Gamma$-graded vector space $\g\oplus {\mathbb K}$ provided with the even linear map $\alpha_{\g}\oplus id$ and the following product 
$$[x + \lambda, y + \eta] = {[x,y]}_{\g} + B(D(x),y)$$ 
is a color Hom-Lie algebra. 
\end{cor}


\subsection{$T^*$-extensions of color Hom-Lie algebras}
We have the following main Theorem.
\begin{thm}
Let $(\mathfrak{g},[\cdot,\cdot]_\g,\alpha,\varepsilon)$ be a color Hom-Lie algebra and $\omega: \g \times \g \longrightarrow \g^*$ be an even bilinear map.  Assume that the coadjoint representation exists. The $\Gamma$-graded vector space $\mathfrak{g}\oplus\mathfrak{g}^*$, where ${(\g\oplus \g^*)}_{\gamma} = {\g}_{\gamma}\oplus {\g^*}_{\gamma}$ for $\gamma\in \Gamma$, provided with the following bracket and  linear map defined respectively by 
\begin{equation}\label{prodT*Ext}
[x+f,y+g]=[x,y]_\g+\omega(x,y)+\pi(x)g -\varepsilon(x,y)\pi(y)f,
\end{equation}
 $$\Omega(x+f)=\alpha(x)+f\circ\alpha, $$
where $\pi$ is the coadjoint representation of $\g$, is a color Hom-Lie algebra if and only if $w\in \mathscr{Z}^2(\g,\g^*)$. In this case, we call $\g \oplus \g^*$ the $T^*$-extension of $\g$ by means of $\omega$. 

If $B_{\g}$ is an invariant scalar product on $\g$, then 
\begin{enumerate}
\item the coadjoint representation exists,
\item  the $T^*$-extension of $\g$ admits also an invariant scalar product $B$ defined by  
\begin{align}\label{bilforT*Ex}
  B :(\mathfrak{g}\oplus\mathfrak{g}^*)^{\otimes2}  &\rightarrow \K \\ \nonumber
     (x+f,y+g)  &\mapsto B(x,y)+f(y)+\varepsilon(x,y) g(x).
\end{align}
\end{enumerate}
\end{thm}

\begin{proof}
For any  homogeneous elements $(x,f),\ (y,g),\ (z,h)$ in $\mathfrak{g}\oplus\mathfrak{g}^*$ we have 
\begin{eqnarray*}
\varepsilon(z,x)[\Omega(x+f),[y+g,z+h]]
&=&\varepsilon(z,x)[\alpha(x)+f\circ\alpha,[y,z]_\g+\omega(y,z)+\pi(y)h-\varepsilon(y,z)\pi(z)g]\\
&=&\varepsilon(z,x)[\alpha(x),[y,z]_\g]_\g+ \varepsilon(z,x)\omega(\alpha(x),[y,z]_\g)\\ & +& \varepsilon(z,x)\pi(\alpha(x))\omega(y,z)     
+\varepsilon(z,x)\pi(\alpha(x)) (\pi(y)h)\\ & -&\varepsilon(z,x)\varepsilon(y,z)\pi(\alpha(x))(\pi(z)g) -\varepsilon(x,y)\pi([y,z]_\g)f\circ \alpha.
\end{eqnarray*}

Consequently,
$$\circlearrowleft_{x,y,z} \varepsilon(z,x)\big[\Omega(x+f),[y+g,z+h]\big] = 0$$
if and only if 
\begin{eqnarray*}
0 &=& \pi(\alpha(x_0))\big( \omega(x_1,x_2)\big) - \varepsilon(x_0,x_1)\pi(\alpha(x_1))\big( \omega(x_0,x_2)\big) + \varepsilon(x_0+x_1,x_2) \pi(\alpha(x_2))\big( \omega(x_0,x_1)\big)\\
&+& \omega(\alpha(x_0), [x_1,x_2]) + \varepsilon(x_1,x_2) \omega([x_0,x_2], \alpha(x_1)) - \omega([x_0,x_1], \alpha(x_2)). 
\end{eqnarray*}
That is $\omega\in \mathscr{Z}^2(\g,\g^*)$.

The proof for the second point is similar to the proof for non-graded case in \cite{benayadi2010hom}.
\end{proof}

\section{ Color Hom-Leibniz algebras arising from  Faulkner construction}

Faulkner construction was introduced  in \cite{faulkner1973geometry}. Recently in \cite{figueroa2009deformations},  it was shown how  Faulkner construction gives rise to a  quadratic Leibniz algebras. In the following, we generalize this result to the case of color Hom-Leibniz algebras.  Color Hom-Leibniz algebras are a non $\varepsilon$-skew-symmetric generalization of color Hom-Lie algebras.

\begin{definition}
A color Hom-Leibniz algebra is a tuple $(\g, [\cdot,\cdot], \alpha, \varepsilon)$ consisting of a $\Gamma$-graded vector space $\g$, an even bilinear mapping $[\cdot,\cdot]$, an even linear map $\alpha: \g \longrightarrow \g$ and a  bicharacter $\varepsilon$ such that 
 $$\varepsilon(z,x) \big[\alpha(x), [y,z]\big] +  \varepsilon(x,y) \big[\alpha(y), [z,x]\big] + \varepsilon(y,z) \big[\alpha(z),[x,y] \big] = 0, \ \ \forall x,y,z\in \g.$$ \\
\end{definition}

\begin{proposition}
Let $({\g}, [\cdot, \cdot ], \alpha, \varepsilon )$ be a multiplicative color Hom-Lie algebra. If $\rho_1$ and $\rho_2$ are two representations of $\g$ on $(M_1,\beta_1)$ and $(M_2,\beta_2)$ respectively, then $\rho_1\otimes \rho_2$ is a representation of $\g$ on $(M_1\otimes M_2, \beta_1\otimes \beta_2)$ defined for  homogeneous elements $m_1\otimes m_2$  by  

 $$(\rho_1\otimes \rho_2)(x)(m_1\otimes m_2) \,=\, \rho_1(x)(m_1)\otimes \beta_2(m_2) \ +\,  \ \varepsilon(x,m)\  \beta_1(m_1)\otimes \rho_2(x)(m_2).$$\\
 
\end{proposition}
 
\begin{proof}

We  check for all $x,y$ in $ \g$ the following two identities:
\begin{eqnarray*}&&\big( \beta_1\otimes\beta_2 \big)\big( (\rho_1\otimes\rho_2)(x)(m_1\otimes m_2) \big) = \big( \rho_1\otimes \rho_2 \big)\big(\alpha(x) \big)\big((\beta_1\otimes\beta_2)(m_1\otimes m_2) \big)\\&&
\big(\rho_1\otimes \rho_2 \big)(\alpha(x))\circ \big(\rho_1\otimes \rho_2 \big)(y) - \ \varepsilon(x,y)\ \big(\rho_1\otimes \rho_2 \big)(\alpha(y))\circ \big(\rho_1\otimes \rho_2 \big)(x) = \big(\rho_1\otimes \rho_2 \big)([x,y])\circ \big(\beta_1\otimes \beta_2 \big).
\end{eqnarray*}

Let $m_1\in M_1$ and $m_2\in M_2$ be two homogeneous elements.
\begin{eqnarray*}
\big(\beta_1\otimes \beta_2 \big) \big( (\rho_1\otimes \rho_2)(x)(m_1\otimes m_2)\big) &=& (\beta_1\otimes \beta_2) \big(\rho_1(x)(m_1)\otimes \beta_2(m_2) \ +\,  \ \varepsilon(x,m)\  \beta_1(m_1)\otimes \rho_2(x)(m_2) \big)\\
&=& \rho_1(\alpha(x))(\beta_1(m_1))\otimes {\beta_2}^2(m_2) \ +\, \ \varepsilon(x,m)\  {\beta_1}^2(m_1)\otimes \rho_2(\alpha(x))(\beta_2(m_2))\\
&=& (\rho_1\otimes \rho_2)(\alpha(x))((\beta_1\otimes \beta_2)(m_1\otimes m_2)).
\end{eqnarray*}  
And 
\begin{eqnarray*}
&&\big[ \big(\rho_1\otimes \rho_2 \big)(\alpha(x))\circ \big(\rho_1\otimes \rho_2 \big)(y)\  - \ \varepsilon(x,y)\ \big(\rho_1\otimes \rho_2 \big)(\alpha(y))\circ (\rho_1\otimes \rho_2 \big)(x)\big](m_1\otimes m_2)\\
 &=& \rho_1(\alpha(x))\big( \rho_1(y)(m_1) \big)\otimes \beta_2^2(m_2)\  +\  \varepsilon(x, y+m_1)\ \beta_1 \big(\rho_1(y)(m_1) \big)\otimes \rho_2(\alpha(x))(\beta_2(m_2))\\
 &+& \varepsilon(y,m_1)\ \rho_1(\alpha(x))(\beta_1(m_1))\otimes \beta_2(\rho_2(y)(m_2)) \ + \  \varepsilon(x+y,m_1)\  {\beta_1}^2(m_1)\otimes \rho_2(\alpha(x))\big( \rho_2(y)(m_2) \big)\\
 &-& \varepsilon(x,y)\  \rho_1(\alpha(y)) \big(\rho_2(x)(m_1) \big)\otimes {\beta_2}^2(m_2) \  -\  \varepsilon(y,m)\ \beta_1 \big(\rho_1(x)(m_1) \big)\otimes \rho_2(\alpha(y))(\beta_2(m_2))\\
 &-& \varepsilon(x,y+m)\ \rho_1(\alpha(y))(\beta_1(m_1))\otimes \beta_2 \big( \rho_2(x)(m_2) \big) \ -\  \varepsilon(x,y+m_1) \varepsilon(y,m_1)\  {\beta_1}^2(m)\otimes \rho_2(\alpha(y)) \big( \rho_2(x)(m_2) \big).\\
\end{eqnarray*}

By virtue of (\ref{RepreIdent}) and (\ref{condition-two-of-multiplicative-representation}), we obtain  

\begin{eqnarray*}
&&\big[ \big(\rho_1\otimes \rho_2 \big)(\alpha(x))\circ \big(\rho_1\otimes \rho_2 \big)(y)\  - \ \varepsilon(x,y)\ \big(\rho_1\otimes \rho_2 \big)(\alpha(y))\circ (\rho_1\otimes \rho_2 \big)(x)\big](m_1\otimes m_2)\\
&=& \rho_1([x,y])(\beta_1(m_1))\otimes {\beta_2}^2(m_2) \ + \ \varepsilon(x+y,m_1){\beta_1}^2(m_1) \otimes \rho_2 ([x,y])(\beta_2(m_2))\\
&=& \big[ (\rho_1\otimes \rho_2)([x,y])\circ (\beta_1\otimes \beta_2) \big] (m_1\otimes m_2).
\end{eqnarray*}

\end{proof}

Now, let $(\mathfrak{g},[\cdot,\cdot ], \alpha,  \varepsilon, B)$ be a quadratic multiplicative color Hom-Lie algebra and  $\rho$ be a representation of $\g$ on $(M,\beta)$, such that both $\alpha$ and $\beta$ are two involutions. Assume that $\tilde{\rho}$ is a representation of $\g$ on $(M^*, \tilde{\beta})$, where $M^*$ is the dual space of $M$.  We denote the dual pairing between $M$ and $M^*$ by  $\langle-,-\rangle$. Recall that $\tilde{\rho}(x)(f) := - \varepsilon(x,f) f \circ \rho(x) $, $\forall x\in\g$, $f\in M^*$ and which means that $\langle m , \tilde{\rho}(x)(f) \rangle = - \varepsilon(m,x) \langle \rho(x)(m), f\rangle$, $\forall m\in M$. For all $m\in M$ and $f\in M^*$, we define the even linear map $\mathscr{D}:M\otimes M^* \longrightarrow \g$ as follows\\
\begin{equation}\label{sc}
B(x,\mathscr{D}(m\otimes f))=\langle \rho(\alpha(x))(m), f\rangle= \varepsilon(x + m,f)\  f( \rho(\alpha(x))(m)),\ \ \textrm{for\ all} \ x\in\mathfrak{g}.
\end{equation}

\begin{lem}
The even linear map $\mathscr{D}:M\otimes M^* \longrightarrow \g$ is a morphism of $\g$-Hom-modules, that is 
$$[x, \mathscr{D}(m\otimes f)] = \mathscr{D}(\rho(x)(m)\otimes \tilde{\beta}(f) \ +\  \varepsilon(x,m)\  \beta(m)\otimes \tilde{\rho}(x)(f)), \ \ \ \forall x\in\g,\ m\in M, \ f\in M^*.$$
\end{lem}

\begin{proof}

Let $m$ and $f$ be  homogeneous elements in respectively $M$ and $M^*$ and $y$ be a homogeneous element in $\g$. 
\begin{eqnarray*}
&&B(y, \mathscr{D} \big(\rho(x)(m)\otimes \tilde{\beta}(f) \ +\  \varepsilon(x,m)\ \beta(m)\otimes \tilde{\rho}(x)(f)\big) )\\
 &=&  B(y, \mathscr{D} \big(\rho(x)(m)\otimes \tilde{\beta}(f)\big) ) \ +\  \varepsilon(x,m)\  B(y, \mathscr{D} \big(\beta(m)\otimes \tilde{\rho}(x)(f) \big))\\
 &=&  \langle \rho \big(\alpha(y) \big) \big(\rho(x)(m) \big), \tilde{\beta}(f)\rangle \ +\  \varepsilon(x,m)\  \langle \rho \big(\alpha(y) \big) \big(\beta(m) \big), \tilde{\rho}(x)(f)\rangle\\
 &=& \langle \beta \big[\rho\big(\alpha(y) \big) \big(\rho(x)(m) \big)\big], f\rangle \ - \  \varepsilon(x,m)\varepsilon(y+m, x)\  \langle \rho(x)\big[\rho \big(\alpha(y) \big) \big(\beta(m) \big)\big], f\rangle\\
 &=& \langle \rho(y) \big(\rho(\alpha(x))(\beta(m)) \big) \ -\  \varepsilon(y,x)\  \rho(x) \big(\rho(\alpha(y))(\beta(m)) \big),f   \rangle\\
 &=& \langle \big(\rho(\alpha[y,x])\circ \beta \big) (\beta(m)), f\rangle \\
 &=& B([y,x], \mathscr{D}(m\otimes f))\\
 &=& B(y, [x, \mathscr{D}(m\otimes f)]). 
\end{eqnarray*}
 
Consequently, $\mathscr{D}$ is a morphism of $\g$-Hom-modules.

\end{proof}

In \cite{de2009lie}, it was proved  that if $(\rho, M)$ is a faithful representation, then $\mathscr{D}$ is surjective. In the sequel, we assume  that $\rho$ on $(M,\beta)$ is a faithful representation. Since $\alpha $ is invertible, then $\mathscr{D}$ is still surjective. Using this argument, it follows that the fact that $\mathscr{D}$ is a morphism of $\g$-Hom-modules is equivalent to\\ 
\begin{equation}\label{equivalent-morphism-of-g-modules}
[\mathscr{D}(m\otimes f),\mathscr{D}(m'\otimes f')]=\mathscr{D}(\rho\big(\mathscr{D}(m\otimes f)\big)(m')\otimes \tilde{\beta}(f')) \ +\  \varepsilon(m + f,m')\  \mathscr{D}(\beta(m')\otimes \tilde{\rho}\big(\mathscr{D}(m\otimes f)\big)(f')).
\end{equation}\\

\begin{proposition}
The $\Gamma$-graded vector space $M\otimes M^*$ provided with  the following bracket 
\begin{equation}\label{bbb}[m\otimes f,m'\otimes f']=\rho \big(\mathscr{D}(m\otimes f) \big)(m')\otimes \tilde{\beta}(f') \ +\ \varepsilon(m+f, m') \ \beta(m')\otimes \tilde{\rho} \big(\mathscr{D}(m\otimes f) \big)(f')
\end{equation}
and the even linear map $\beta\otimes \tilde{\beta}$ is a color Hom-Leibniz algebra.\\ 
\end{proposition}

\begin{proof}
For  homogeneous elements $m\otimes f$, $m'\otimes f'$ and $m''\otimes f''$ in $M\otimes M^*$, we have:

\begin{eqnarray*}
&& \big[\beta(m)\otimes \tilde{\beta}(f), [m'\otimes f', m''\otimes f'']\big] = \rho \big(\mathscr{D}(\beta(m)\otimes \tilde{\beta}(f)) \big) \big[\rho\big( \mathscr{D}(m'\otimes f') \big) (m'') \big] \otimes {\tilde{\beta}}^2(f'')\\ && \quad
+ \varepsilon(m+f , m' + m'' +f') \beta\big[ \rho\big(\mathscr{D}(m'\otimes f')\big)(m'') \big]\otimes \tilde{\rho} \big( \mathscr{D}(\beta(m)\otimes \tilde{\beta}(f)) \big) (\tilde{\beta}(f''))\\ && \quad
+ \varepsilon(m'+f', m'')\  \rho \big( \mathscr{D}(\beta(m)\otimes \tilde{\beta}(f)) \big) (\beta(m''))\otimes \tilde{\beta} \big[ \tilde{\rho} \big( \mathscr{D}(m'\otimes f') \big)(f'') \big]\\ && \quad
+ \varepsilon(m+f+m'+f' , m''){\beta}^2( m'')\otimes \tilde{\rho}\big( \mathscr{D}(\beta(m)\otimes \tilde{\beta}(f))\big) \big[ \tilde{\rho} \big((\mathscr{D}(m'\otimes f')\big)(f') \big].
\end{eqnarray*}

\begin{eqnarray*}
 && \big[\beta(m')\otimes \tilde{\beta}(f') , [m\otimes f, m''\otimes f'']\big] = - \rho\Big(\mathscr{D}(\beta(m')\otimes \tilde{\beta}(f'))\Big) \Big[\rho \Big( \mathscr{D}(m\otimes f)\Big) (m'')\Big]\otimes {\tilde{\beta}}^2(f'')\\ && \quad
 - \varepsilon(m' + f' , m+ f+ m'')\  \beta \Big[\rho\Big(\mathscr{D}(m\otimes f)\Big) (m'') \Big]\otimes \tilde{\rho}\Big( \mathscr{D}( \beta(m')\otimes \tilde{\beta}(f')) \Big) (\tilde{\beta}(f''))\\ && \quad
 - \varepsilon(m+f,m'')\  \rho\Big(\mathscr{D}(\beta(m')\otimes \tilde{\beta}(f'))\Big) (\beta(m''))\otimes \tilde{\beta} \Big[\tilde{\rho}\Big(\mathscr{D}(m\otimes f) \Big)(f'') \Big]\\ && \quad
 -\varepsilon(m+f +m'+f', m'')\  \beta^2(m'')\otimes \tilde{\rho} \Big( \mathscr{D}(\beta(m')\otimes \tilde{\beta}(f')) \Big) \Big[ \tilde{\rho} \Big( \mathscr{D}(m\otimes f) \Big)(f'')\Big].
\end{eqnarray*}

\begin{eqnarray*}
&&  \big[[m\otimes f, m'\otimes f'], \beta(m'')\otimes \tilde{\beta}(f'')\big] 
 =\rho \Big[ \mathscr{D} \Big( \rho \big(\mathscr{D}(m\otimes f) \big)(m')\otimes \tilde{\beta}(f') \Big) \Big] (\beta(m''))       \otimes {\tilde{\beta}}^2(f'')\\ && \quad
+ \varepsilon(m+f+m'+f', m'') \beta^2(m'')\otimes  \tilde{\rho} \Big[ \mathscr{D} \Big( \rho \big( \mathscr{D}(m\otimes f) \big) (m')\otimes \tilde{\beta}(f')) \Big) \Big] \tilde{\beta}(f'')\\ && \quad
+ \varepsilon(m+f,m')\ \rho \Big[ \mathscr{D} \Big( \beta(m')\otimes \tilde{\rho} \big( \mathscr{D}(m\otimes f )\big)(f') \Big)  \Big] \beta(m'') \otimes {\tilde{\beta}}^2(f'')\\ && \quad
+ \varepsilon(m+f,m')\varepsilon(m+f+m'+f', m'') \beta^2(m'')\otimes \tilde{\rho} \Big[ \mathscr{D} \Big(\beta(m')\otimes \tilde{\rho}\big( \mathscr{D}(m\otimes f) \big)(f') \Big) \Big] \tilde{\beta}(f'')
\end{eqnarray*}

By using  condition (\ref{equivalent-morphism-of-g-modules}), we deduce that 

\begin{eqnarray*}
\big[\beta(m)\otimes \tilde{\beta}(f), [m'\otimes f', m''\otimes f'']\big]  - \varepsilon(m+f, m'+f') \big[\beta(m')\otimes \tilde{\beta}(f') , [m\otimes f, m''\otimes f'']\big]\\ - \big[[m\otimes f, m'\otimes f'], \beta(m'')\otimes \tilde{\beta}(f'')\big] =0.
\end{eqnarray*} 
\end{proof}

\begin{cor}
Consider the Hom-Leibniz algebra $(M\otimes M^*, [,], \beta\otimes \tilde{\beta})$ defined above. If the morphism of $\g$-Hom-modules $\mathscr{D}$ is bijective, then $(M\otimes M^*, [,], \beta\otimes \tilde{\beta})$ is a quadratic Hom-Leibniz algebra such that the quadratic structure is given by $B: M\otimes M^* \times M\otimes M^* \longrightarrow \K$ such that 
\begin{eqnarray*} 
  B(m\otimes f, m'\otimes f')= B_{\g}(\mathscr{D}(m\otimes f), \mathscr{D}(m'\otimes f')). 
\end{eqnarray*}
\end{cor}

\begin{proof}
By straightforward  computations, we prove that $B$ is $\varepsilon$-symmetric, $\beta\otimes \tilde{\beta}$ is $B$-symmetric and $B$ is invariant.  
\end{proof}


\end{document}